\documentclass[12pt]{article}

\usepackage[utf8]{inputenc}   
\usepackage[T1]{fontenc}      
\usepackage[english]{babel}   
\usepackage{amsmath, amssymb} 
\usepackage{amsthm}           
\usepackage{geometry}         
\usepackage{graphicx}
\usepackage{caption}
\usepackage{subcaption}

\usepackage{cite}
\usepackage[colorlinks=true, linkcolor=blue, citecolor=blue]{hyperref}

\newtheorem{theorem}{Theorem}[section]
\newtheorem{lemma}[theorem]{Lemma}

\newtheorem{definition}[theorem]{Definition}
\newtheorem{proposition}[theorem]{proposition}

\title{Existence and Uniqueness of Local Solutions for a Class of Partial Differential-Algebraic Equations}
\author{
  Seyyid Ali Benabdallah and Messaoud Souilah\\  
  \textit{Dynamic Systems Laboratory (DSL),Department of Mathematics} \\  
  \textit{University of Science and Technology Houari Boumediene} \\  
  \textit{Algiers, Algeria} \\  
  \texttt{sbenabdallah@usthb.dz} \\[1em]  
  \\
  \textit{Dynamic Systems Laboratory (DSL),Department of Mathematics} \\  
  \textit{University of Science and Technology Houari Boumediene} \\  
  \textit{ Algiers, Algeria} \\  
  \texttt{msouilah@usthb.dz}  
}

\date{\today}

\begin{document}

\maketitle

\abstract{In this work, we present a result on the local existence and uniqueness of solutions to nonlinear Partial Differential-Algebraic Equations (PDAEs). By applying established theoretical results, we identify the conditions that guarantee the existence of a unique local solution. The analysis relies on techniques from functional analysis, semi-group theory, and the theory of differential-algebraic systems. Additionally, we provide applications to illustrate the effectiveness of this result.}

\maketitle

\section{Introduction} 

  In this paper, we study a specific type of partial differential-algebraic equations (PDAEs) defined as follows:
\begin{equation}
\left\{
\begin{array}{ll}
V_t(t,x) = F(x, t, U(t,x)), & \\
0 = G(x, t, U(t,x)), & \\
V(0, x) = V_0(x), &
\end{array}
\right.
\label{eq1}
\end{equation}

where $ U = (V, w) $ is an unknown function depending on the spatial variable $ x \in \Omega $ and time $ t $, with $ \Omega \subset \mathbb{R}^n $.

This type of equation models numerous physical and chemical phenomena. For example, the Navier-Stokes equations in fluid mechanics \cite{15} and the transport of ions in electrolytes \cite{14} fall within this framework. For a broader overview of partial differential-algebraic equations (PDAEs), we refer the reader to \cite{2}, \cite{9}, \cite{11}, \cite{16}, and \cite{17}.

In this work, we focus specifically on a class of PDAEs that can be formulated as a problem in a Banach space $ \mathcal{X} $, enabling us to apply existing results to establish the existence and uniqueness of solutions.

Let $ (\mathcal{E}, \| \cdot \|_{\mathcal{E}}) $, $ (\mathcal{S}, \| \cdot \|_{\mathcal{S}}) $, and $ (\mathcal{H}, \| \cdot \|_{\mathcal{H}}) $ be three Banach spaces, and define $ \mathcal{X} = \mathcal{E} \times \mathcal{S} $ with the norm

\[
\| U \|_{\mathcal{X}} = \| V \|_{\mathcal{E}} + \| w \|_{\mathcal{S}},
\]

where $ U = (V, w) $, $ V \in \mathcal{E} $, and $ w \in \mathcal{S} $.

If $ G(t,x, U(t,x)) = \mathcal{L}(w(t,x)) - \mathcal{G}(t,x, V(t)) $, we can rewrite the initial-boundary-value problem \eqref{eq1} as an abstract problem in Banach space $ \mathcal{X} $:

\begin{equation}
\left\{
\begin{array}{ll}
V_t(t) = \mathcal{A}(V(t)) + \mathcal{F}(t, U(t)), & \\
0 = \mathcal{L}(w(t)) - \mathcal{G}(t, V(t)), & t \geq 0, \\
V(0) = V_0, &
\end{array}
\right.
\label{eq2}
\end{equation}

where $ \mathcal{A} : \mathcal{E} \to \mathcal{E} $, $ \mathcal{F} : \mathbb{R}_+ \times \mathcal{X} \to \mathcal{E} $, $ \mathcal{L} : \mathcal{S} \to \mathcal{H} $, and $ \mathcal{G} : \mathbb{R}_+ \times \mathcal{E} \to \mathcal{H} $.

By a solution, we mean the functions $ V(t) $ and $ w(t) $, taking values in $ \mathcal{E} $ and $ \mathcal{S} $, respectively, such that $ V(t) $ and $ w(t) $ are continuous for $ t \geq 0 $.

This paper is organized as follows. Section 1 presents the preliminaries necessary for understanding the problem. In Section 2, we establish the existence and uniqueness theorems for the local solution of a class of partial differential-algebraic equations (PDAEs). Finally, Section 3 illustrates our theoretical results through concrete applications.
\section{Preliminary}
To make it easier for the reader, we will introduce the Banach spaces $\mathcal{X}$ and explain the key concepts and theorems that we will use.

\begin{theorem} \label{T1}\cite{1}
 Let $\mathcal{F}:[0,+\infty[\times \mathcal{X} \longrightarrow \mathcal{X}$ be continuous in $t$ for $t\geq 0$ and locally Lipchitz continuous in $U$, uniformly in $t$ on bounded intervals. If $\mathcal{A}$ is the
infinitesimal generator of a $C_0$ semi-group $T(t)$ on $\mathcal{X}$ then for every $U_0 \in  \mathcal{X}$ there is a $t _{max}\leq +\infty$ such that the initial value problem
\begin{equation}
\left\{
\begin{array}{lcc}
\frac{dU(t)}{dt}=\mathcal{A}U(t)+\mathcal{F}(U(t)), &  &  \\ 
&  & \ t\geq 0, \\ 
U(0)=U_{0}, &  & 
\end{array}
\right.  \label{Tp2}
\end{equation}
has a unique mild solution $U$ on $[0,t_{max}[$. Moreover, if $t _{max} \leq +\infty$ then
$$\lim_{t\rightarrow t_{max}}\|U(t)\|_{\mathcal{X}}=+\infty.$$
\end{theorem}
\begin{definition}\cite{1}
    Let $T(t)$ be a $C_{0}$ semi-group of contractions on $\mathcal{X}$. A continuous function $U : [0,T] \to \mathcal{X}$ is a mild solution to the problem \eqref{Tp2} if it satisfies
$$ U (t) = T (t) U_0 + \int_0^t T (t - s) \mathcal{F} (U (s)), ds, \quad 0 \leq t \leq T,
$$
where $\mathcal{F}$ is locally Lipchitz on $\mathcal{X}$.
\end{definition}
\section{Existence and Uniqueness}
\begin{theorem}\label{t5}
 Let $\mathcal{F}: \mathbb{R}_+ \times \mathcal{X} \to \mathcal{E}$ and $\mathcal{G}: \mathbb{R}_+ \times \mathcal{E} \to \mathcal{H}$ be continuous in $t$ for $t \geq 0$, uniformly in $t$ on bounded intervals, and locally Lipschitz continuous in $U$ and $V$, respectively. Let $\mathcal{L} : \mathcal{S} \to \mathcal{H}$ be continuous, linear, and invertible. If $\mathcal{A}$ is the infinitesimal generator of a $C_0$-semigroup $T(t)$, $t \geq 0$, on $\mathcal{E}$, then for every $V_0 \in \mathcal{E}$, there exists $t_{\max} \leq +\infty$ such that the initial value problem

\begin{equation}
\left \{\begin{array}{lll}
V_t(t) =\mathcal{A}(V(t))+\mathcal{F}(t,V(t),w(t)), \textbf{          }t\geq 0 \\
0=\mathcal{L}(w(t))-\mathcal{G}(t,V(t))  \\
V(0)= V_0.
\end{array}
\right.\label{eq22}
\end{equation}
has a unique solution $U\in \mathcal{C} ([0, t max [, \mathcal{X})$. 
\end{theorem}
\begin{proof}
By the invertibility of $\mathcal{L}$, we have $w(t) = \mathcal{L}^{-1} \mathcal{G}(t, V(t))$, where $w(t) \in \mathcal{S}$ and $V(t) \in \mathcal{E}$ for $t \geq 0$. Substituting $w(t)$ into the first equation, we obtain   
\begin{equation}
\left \{\begin{array}{lll}
V_t(t) =\mathcal{A}(V(t))+\mathcal{F}(t,V(t),\mathcal{L}^{-1}\mathcal{G}(t,V(t)))  , \textbf{          }t\geq 0\\
   V(0)=V_0 .
   \end{array}
   \right.\label{eq3}
\end{equation}
We define the operator $\mathcal{K}$ as follows:
$$ \begin{array}{lclll}
   \mathcal{K}:&\mathbb{R}_+\times \mathcal{E}&  & \longrightarrow & \mathcal{E} \\\\
    & (t,V(t)) &  & \longmapsto & \mathcal{K}(t,V(t))= \mathcal{F}(t,V(t),\mathcal{L}^{-1}\mathcal{G}(t,V(t)))
   \end{array} $$
So, the problem \eqref{eq3} take the form 
\begin{equation}
\left \{\begin{array}{lll}
   V_t(t) =\mathcal{A}(V(t))+\mathcal{K}(t,V(t)) , \textbf{          }t\geq 0 \\
    V(0)= V_0.
   \end{array}
   \right.\label{eq4}
\end{equation}
Thus, the problem is defined on $\mathcal{E}$. To resolve problem \eqref{eq22}, it suffices to solve problem \eqref{eq4}. Since $\mathcal{E}$ is a Banach space and $\mathcal{A}$ generates a $C_0$-semigroup on $\mathcal{E}$, it suffices to prove that $\mathcal{K}$ is continuous in $t$ for $t\geq 0$, locally Lipschitz on $\mathcal{E}$, and uniformly Lipschitz in $t$ on bounded intervals.

We have $\mathcal{L}$ is continuous, linear, and invertible, and that both $\mathcal{F}$ and $\mathcal{G}$ are continuous in $t$ for $t\geq 0$, therefore, $\mathcal{G}$ is continuous in $t$ for $t\geq 0$.

Let $t'\geq 0$, $C\geq 0$ and $V_1(t),V_2(t) \in \mathcal{E}$ with $\|V_1(t)\|_{\mathcal{E}}\leq C$, $\|V_2(t)\|_{\mathcal{E}}\leq C$ and $t\in [0,t']$, so 
\begin{equation*}
 \begin{array}{lll}
   \|\mathcal{K}(t,V_1(t))-\mathcal{K}(t,V_2(t))\|_{\mathcal{E}} &=&\|F(t,V_1(t),L^{-1}\mathcal{G}(V_1(t)))-F(t,
   V_2,L^{-1}\mathcal{G}(t,V_2(t)))\|_{\mathcal{E}}. 
   \end{array}
\end{equation*}
Since $\mathcal{F}$ is locally Lipchitz in $U$ and uniformly in $t$ on bounded intervals, then there exists a positive constant 
$L(t',C)\geq 0$ that depends on $t'$ and $C$ such that
\begin{equation*}
 \begin{array}{lll}
   \|\mathcal{K}(t,V_1(t))-\mathcal{K}(t,V_2(t))\|_{\mathcal{E}} &\leq& L(t',C)\|(V_1(t),\mathcal{L}^{-1}\mathcal{G}(t,V_1(t)))-(
   V_2(t),\mathcal{L}^{-1}\mathcal{G}(t,V_2(t)))\|_{\mathcal{X}}, \\ \\
   &=&L(t',C)(\|V_1(t) -V_2(t)\|_{\mathcal{E}}+\|\mathcal{L}^{-1}\mathcal{G}(t,V_1(t))-\mathcal{L}^{-1}\mathcal{G}(t,V_2(t))\|_{\mathcal{S}}).\\\\
   \end{array}
\end{equation*}
We have that $\mathcal{L}$ is a continuous, linear, and invertible operator between two Banach spaces. Then, $\mathcal{L}^{-1}$ is also linear and continuous. Therefore, there exists a positive constant $L_1(C) \geq 0$, which depends on $C$, such that

\begin{equation*}
 \begin{array}{lll}
   \|\mathcal{K}(t,V_1(t))-\mathcal{K}(t,V_2)\|_{\mathcal{E}} &\leq& L(t',C)(\|V_1(t) -V_2(t)\|_{\mathcal{E}}+L_1(C)\|\mathcal{G}(t,(V_1(t))-\mathcal{G}(t,V_2(t)))\|_{\mathcal{H}}).
\end{array}
\end{equation*}
As $\mathcal{G}$ is locally Lipchitz in $V$ and uniformly in $t$ on bounded intervals, then there exists a positive constant $L_2(t',C)\geq 0$ that depends on $t'$ and $C$ such that
\begin{equation*}
 \begin{array}{lll}
   \|\mathcal{K}(t,V_1(t))-\mathcal{K}(t,V_2(t))\|_{\mathcal{E}} &\leq& L(t',C)(\|V_1(t) -V_2(t)\|_{\mathcal{E}}+L_1(C)L_2 (t',C)\|V_1(t)-V_2(t)\|_{\mathcal{E}}),\\\\
   &\leq & L(t',C)(1+L_1(t',C)L_2 (t',C)) \|V_1(t)-V_2(t)\|_{\mathcal{E}}.
\end{array}
\end{equation*}
Then,
$$\|\mathcal{K}(t,V_1(t))-\mathcal{K}(t,V_2(t))\|_{\mathcal{E}} \leq L_3(t',C)\|V_1(t)-V_2(t)\|_{\mathcal{E}}.$$
where $L_3(t',C)= L(t',C)(1+L_1(C)L_2 (t',C))$.
So, $\mathcal{K}$ is continuous in $t$ for $t\geq 0$ and locally Lipchitz in $\mathcal{E}$ and uniformly in $t$ in bounded intervals. Therefore, by Theorem \ref{T1}, there exists a positive $t_{max}\leq +\infty $ such that problem \eqref{eq4} admits a unique mild solution $V\in\mathcal{C}([0,t_{max}[,\mathcal{E})$. Moreover, if $t_{max}<+\infty$ then
$$\lim_{t \to t_{max}} \|V (t)\|_0=+\infty.$$
By the continuity of $\mathcal{L}^{-1}$ and $\mathcal{G}$, there exists a unique $w$ that satisfies
$$w(t)=\mathcal{L}^{-1}\mathcal{G}(t,V(t)) \textit{,    } w\in \mathcal{C}([0,t_{max}[,\mathcal{S}).$$
Therefore, problem \eqref{eq22} has a unique solution $U=(V,w)$  that satisfies
$$U \in \mathcal{C}([0,t_{max}[,\mathcal{X}).$$
\end{proof}

\section{Application }
      In this section, we apply the theorem to non-linear problem.

Let $\mathcal{P'}$ be a nonlinear problem defined as follows.
   \begin{equation}  \label{ex2}
\tag{$\mathcal{P'}$}
\left \{ 
\begin{array}{lcc}
u_{t}=u_{xx}+wv+uw, &  &  \\ 
v_{t}= v_{xx}+wv+u, &  & t\geq 0 \textit{  and  } x\in \Omega,\\ 
0= w_{xxxx}+w+u+v,&  & 
\end{array}
\right.
\end{equation}
where $\Omega$ is a bounded open set in $\mathbb{R}$ and $U=(u,v,w)$ is an unknown function.

The initial conditions are defined as follows:
$$
 u(0,x)=u_{0}\ ,\ v(0,x)=v_{0}, 
$$
and boundary conditions 
$$u_x|_{\partial\Omega}= v_x|_{\partial\Omega}=0.$$

 We introduce the Hilbert spaces
 $$\mathcal{X}=L^2(\Omega)\times L^2(\Omega)\times H^2_0(\Omega) \textit {   and   } \mathcal{E}=L^2(\Omega)\times L^2(\Omega),$$
 with the inner products given by
 $$ \langle U_1,U_2\rangle_{\mathcal{X}}=\int_{\Omega } u_1  u_2+\int_{\Omega } v_1  v_2+\int_{\Omega } w_{1xx} w_{2xx} \textit{    and    } \langle V_1,V_2\rangle_{\mathcal{E}}=\int_{\Omega }u_1 u_2+\int_{\Omega } v_1 v_2,$$ and the norms $$\|U\|_{\mathcal{X}}=\left(\int_{\Omega } u^2+\int_{\Omega } v^2+\int_{\Omega } w^2_{xx}\right) ^{\frac{1}{2}}\textit{    and    } \|U\|_{\mathcal{E}}=\left(\int_{\Omega } u^2+\int_{\Omega } v^2\right) ^{\frac{1}{2}},$$
 where $ U_1,U_2 \in \mathcal{X}$ ,$V_1,V_2 \in \mathcal{E}$ and $U_1=(V_1,w_1)$, $U_2=(V_2,w_2)$, $V_1=(u_1,v_1)$ and $V_2=(u_2,v_2).$ 
 
We reformulate the problem (\ref{ex2}) along with the initial condition of $U$ as an abstract problem in the Hilbert space 
$\mathcal{X}$ as follows:
\begin{equation}
\left \{\begin{array}{lll}
V_t(t) =\mathcal{A}(V(t))+\mathcal{F}(U(t)),& &  \\
0=\mathcal{L}(w(t))-\mathcal{G}(V(t))& & \textbf{                                        }t\geq 0 , \\
V(0)= V_0 ,
\end{array}
\right.\label{eq6}
\end{equation}
where $\mathcal{A}:\mathcal{D}(A)\subset \mathcal{E}\rightarrow \mathcal{E},$ $\mathcal{F}:\mathcal{X}\rightarrow\mathcal{E}$, $\mathcal{L}: H^2_0(\Omega)\rightarrow H^{-2}(\Omega)$ and  $\mathcal{G}:\mathcal{E}\rightarrow H^{-2}(\Omega)$ where $\mathcal{D}(\mathcal{A}) $ is a domain of $\mathcal{A}$. These applications are given in the following form:

For each $U=(V,w) \in \mathcal{X}$ such that  $V=(u,v)\in \mathcal{D} (A)$,  
\begin{equation*}
\mathcal{A}\left(V\right)=\left( 
\begin{array}{c}
 u_{xx} \\ 
 v_{xx}
\end{array}
\right),  \mathcal{F}(U)=\left( 
\begin{array}{c}
wv+uw \\ 
wv+u \\ 
\end{array}
\right) , \mathcal{L}(w)= w_{xxxx}+w
\end{equation*}
and $$\mathcal{G}(V)=u+v.$$
The domain of $\mathcal{A}$ is given by 
$$
\mathcal{D}(\mathcal{A})=H^{2}(\Omega)\cap H_0^{1}(\Omega )\times H^{2}(\Omega)\cap H_0^{1}(\Omega ).
$$
\begin{proposition}\label{p5}
The operator $\mathcal{L}: H^2_0(\Omega)\rightarrow H^{-2}(\Omega)$ is continuous, linear, and invertible.
\end{proposition}
\begin{proof}
We decompose this proof into three steps. First, we show that $\mathcal{L}$ is well-defined on $H^2_0(\Omega)$. Then, we prove that it is linear and continuous. Finally, we demonstrate that $\mathcal{L}$ is bijective.

\subsubsection*{Step 1: Show that the operator $\mathcal{L}$ is well-defined on the space $H^2_0(\Omega)$.}
Let $w\in H^2_0(\Omega)$ and $\psi  \in \mathcal{D}(\Omega)$:
\begin{equation}\label{d2}
    \begin{array}{lcl}
      |\langle \mathcal{L} (w),\psi\rangle|&=& |\langle w_{xxxx}+w,\psi\rangle|, \\\\
     &\leq & |\langle w_{xxxx},\psi\rangle|+|\langle w,\psi\rangle|, \\\\ 
          &\leq & |\langle w_{xx},\psi_{xx}\rangle|+|\langle w,\psi\rangle|, \\\\
          &\leq & \|w\|_{H^2_0(\Omega)}\|\psi\|_{H^2_0(\Omega)}+\|w\|_{L^2(\Omega)}\|\psi\|_{L^2(\Omega)},
    \end{array}
\end{equation}
hence $\mathcal{L}(w)$ is well-defined on the space $\mathcal{D}(\Omega)$ for all $w\in H^2_0(\Omega)$.

The linearity is clear, now let us show the continuity. Let $w\in H^2_0(\Omega)$ and $\psi_{n}$ be a sequence convergent to $\psi$ in $\mathcal{D}(\Omega)$, then 
\begin{equation*}
    \begin{array}{lcl}
      |\langle \mathcal{L} (w),\psi_n-\psi\rangle|&=& |\langle w_{xxxx}+w,\psi_n-\psi\rangle|, \\\\
     &\leq & |\langle w_{xxxx},\psi_n-\psi\rangle|+|\langle w,\psi_n-\psi\rangle|, \\\\ 
          &\leq & |\langle  w_{xx},(\psi_n)_{xx}-\psi_{xx}\rangle|+|\langle w,\psi_n-\psi\rangle|, \\\\
          &\leq & \|w_{xx}\|_{L^2(\Omega)}\|(\psi_n)_{xx}-\psi_{xx}\|_{L^2(\Omega)}+\|w\|_{L^2(\Omega)}\|\psi_n-\psi\|_{L^2(\Omega)},\\\\
          &\leq &C(\Omega) \left (\| w_{xx}\|_{L^2(\Omega)}\|(\psi_n)_{xx}-\psi_{xx}\|_{\infty}+\|w\|_{L^2(\Omega)}\|\psi_n-\psi\|_{\infty}\right), 
    \end{array}
\end{equation*}
 where $C(\Omega)> 0$. As $n \rightarrow \infty$, we have $$\lim\limits_{n \rightarrow \infty}|\langle\mathcal{L}(w), \psi_{n}-\psi\rangle|=0, $$
which implies that $\mathcal{L}(w)$ is continuous on $ \mathcal{D}(\Omega)$. Therefore, $\mathcal{L}(w) \in \mathcal{D'}(\Omega)$ for all $w\in H^2_0(\Omega)$.

Let $w\in H^2_0(\Omega)$, by \eqref{d2}  and Poincarre inequality, there exists a positive constant $C > 0$ such that
 \begin{equation}\label{D2}
   |\langle \mathcal{L} (w),\psi\rangle |\leq C\|w\|_{H^2_0(\Omega)}\|\psi\|_{H^2_0(\Omega)},\textit{    for all  }\psi  \in \mathcal{D}(\Omega).  
 \end{equation}
Since $\mathcal{L}(w) \in \mathcal{D}'(\Omega)$ and satisfies \eqref{D2} for all $w \in H^2_0(\Omega)$, it follows that $\mathcal{L}(w) \in H^{-2}(\Omega)$ for all $w \in H^2_0(\Omega)$. Therefore, the operator $\mathcal{L}$ is well-defined.
\subsubsection*{Step 3: Show that the operator $\mathcal{L}$ is linear and continuous.}
Let $w,v\in H^2_0(\Omega)$, then 
\begin{equation*}
    \begin{array}{lcl}
       |\langle \mathcal{L}(w), v \rangle|&=&|\langle  w_{xxxx}+w, v \rangle|,\\\\
       &=&|\int_{\Omega}w_{xxxx}v+\int_{\Omega}wv |,\\\\
       &\leq & |\int_{\Omega}w_{xx} v_{xx}|+|\int_{\Omega}wv |,\\\\
       &\leq & \|w_{xx}\|_{L^2(\Omega)}\|v_{xx}\|_{L^2(\Omega)}+\|w\|_{L^2(\Omega)}\|v\|_{L^2(\Omega)},
    \end{array}
\end{equation*}
by Poincare's inequality, there exists $C>0$ such that
$$|\langle \mathcal{L}(w), v \rangle|\leq C\|w\|_{H^2_0(\Omega)}\|v\|_{H^2_0(\Omega)}. $$
 So
 \begin{equation}
  \|\mathcal{L}(w)\|_{H^{-2}(\Omega)} = \sup_{v \in H^2_0(\Omega), \, \|v\|_{H^2_0(\Omega)} \leq 1} |\langle \mathcal{L}(w), v \rangle|\leq C \|w\|_{H^2_0(\Omega)}.   
 \end{equation}
Therefore, $\mathcal{L}$ is continuous on $H^2_0(\Omega)$.

\subsubsection*{Step 2: Show that the operator $\mathcal{L}$ is bijective.}

 Let $g\in H^{-2}(\Omega)$, such that $\mathcal{L}(w)=g$, then the equation becomes: $$ w_{xxxx}+w=g.$$ 
 Let $\phi \in H^2_0(\Omega)$, Multiplying by $\phi $ and integrating by parts, we obtain
$$
\int_{\Omega }w_{xx} \phi_{xx}+\int_{\Omega }w\phi =\int_{\Omega }g\phi.$$
We define 
\begin{equation*}
    \begin{array}{ccl}
         a:H^2_0(\Omega)\times H^2_0(\Omega)&\longrightarrow &\mathbb{R}  \\\\
         (w,v)&\longmapsto & a(w,\phi )=\int_{\Omega } w_{xx} \phi_{xx}+\int_{\Omega }w\phi
    \end{array}
\end{equation*}
 and 
 \begin{equation*}
     \begin{array}{ccl}
       l:H^2_0(\Omega)&\longrightarrow &\mathbb{R}\\\\
         \phi &\longmapsto &l(\phi)=\int_{\Omega }g\phi.
     \end{array}
 \end{equation*}
Since $a$ is bilinear, it suffices to show that $a$ is continuous and coercive.
\newline
Let $w,\phi \in H^2_0(\Omega)$ 
$$
|a(w,\phi )|=\left|\int_{\Omega } w_{xx}  \phi_{xx}+\int_{\Omega }w\phi \right|. 
$$
By Poincare's lemma, there exists $C>0$ depending only on $\Omega $ such that 
$$
|a(w,\phi )|\leq C\|w\|_{H^2_0(\Omega)}\|\phi \|_{H^2_0(\Omega)}, 
$$
then $a$ is continuous.

So, for $w\in H^2_0(\Omega)$ 
$$
a(w,w)=\left(  \int_{\Omega}  w_{xx}^{2}+\int_{\Omega }w^{2}\right) \geq \|w\|^2_{H^2_0(\Omega)}, 
$$
then $a$ is coercive.

Since $ g \in H^{-2}(\Omega) $, there exists a constant $ C > 0 $ such that  
$$
|\langle g, \psi \rangle| \leq C \|\psi\|_{H^2_0(\Omega)}, \quad \text{for all } \psi \in \mathcal{D}(\Omega).
$$  
Now, for any $ \psi \in H^2_0(\Omega) $, there exists a sequence $ (\psi_n) \subset \mathcal{D}(\Omega) $ such that $ \psi_n \to \psi $ in $ H^2_0(\Omega) $. Therefore,  
$$
|l(\psi)| = |\langle g, \psi \rangle| = \left|\langle g, \lim_{n \to \infty} \psi_n \rangle\right|.
$$ 
By the continuity of $ g $ as a functional in $ H^{-2}(\Omega) $,  
$$
|l(\psi)| \leq \lim_{n \to \infty} |\langle g, \psi_n \rangle| \leq C \lim_{n \to \infty} \|\psi_n\|_{H^2_0(\Omega)} = C \|\psi\|_{H^2_0(\Omega)}.
$$  
Hence, $ l $ is a continuous linear on $ H^2_0(\Omega) $.

Since $l$ is linear continuous and we have $a$ is bilinear, continuous, and coercive, by applying the Lax-Milgram lemma, there existe unique solution $w\in H^2_0(\Omega)$ that satisfies$$
a(w,\phi )=l(\phi ),\  \forall \phi \in {\mathcal{S}}.
$$
Therefore, $\mathcal{L}$ is invertible.
\end{proof}

\begin{proposition}\label{p6}
     The operator $\mathcal{G}:\mathcal{E} \longrightarrow \mathcal{S}$ given by $$\mathcal{G}(V)=u+v$$ is a continuous linear operator.
\end{proposition}
\begin{proof}
Let $V=(u,v)\in \mathcal{E}$, then $\mathcal{G}(V)=(u+v) \in L^2(\Omega)$, hence $\mathcal{G}(V)\in \mathcal{D'}(\Omega)$. By the characterization of $H^{-1}(\Omega)$, it follows that $\mathcal{G}(V)\in  H^{-1}(\Omega)$. Therefore, the operator $\mathcal{G}(V)\in  H^{-2}(\Omega)$.

\begin{equation}
  \|\mathcal{G}(V)\|_{H^{-2}(\Omega)} = \sup_{v \in H^2_0(\Omega), \, \|v\|_{H^2_0(\Omega)} \leq 1} |\langle \mathcal{G}(V), v \rangle|\leq  \|u+v\|_{L^2(\Omega)}\leq  \sqrt{2} \left(\|u\|_{L^2(\Omega)}^2+\|v\|_{L^2(\Omega)}^2\right)^{\frac{1}{2}}.
 \end{equation}
  $\mathcal{G}$ is a continuous linear operator. 
\end{proof}

 Since the domain of $\mathcal{A}$ is dense, it suffices to show that it is dissipative and maximal.

\begin{proposition}\label{p7}
The operator $\mathcal{A}$ is dissipative, that is ${Re}\langle
\mathcal{A}V,V\rangle_{\mathcal{E}}\leq 0.$
\end{proposition}
\begin{proof}
Let $V=(u,v)^{T}\in \mathcal{D}(\mathcal{A})$, then 
\[
\begin{array}{ccl}
\langle\mathcal{A}V,V\rangle_{\mathcal{E}} & = & \int_{\Omega } 
 u_{xx} u+\int_{\Omega } v_{xx}  v,\\ 
&  &  \\ 
& = & -\int_{\Omega }( u_{x})^2-\int_{\Omega }( v_{x})^2, \\ 
&  &  \\ 
& \leq & 0.
\end{array}
\]
So $\mathcal{A}$ is dissipative.
\end{proof}

\begin{proposition}\label{p8} 
The operator $\mathcal{A}$ is maximal, that is $\mathcal{R}(I-%
\mathcal{A})=\mathcal{E}.$
\end{proposition}

\begin{proof}
Let $g=(g_{1},g_{2})^{T}\in \mathcal{E}$. We will prove that there exists $V\in \mathcal{D}(\Omega)$ satisfying
$$
(I-\mathcal{A})V=g\Leftrightarrow \left \{ 
\begin{array}{c}
-  u_{xx}+u=g_{1}, \\ 
\\ 
- v_{xx}+v=g_{2}.
\end{array}
\right. 
$$
By applying the Lax-Milgram lemma, along with the regularity of elliptic problems and the boundary conditions, there exists a unique $ V = (u, v) \in \mathcal{D}(\mathcal{A}) $ such that 
$$
(I - \mathcal{A})V = g.
$$
Consequently, 
$$
\mathcal{R}(I - \mathcal{A}) = \mathcal{E}.
$$
Thus, $ \mathcal{A} $ is maximal.

\end{proof}

\begin{theorem}
\label{t9} The operator $\mathcal{A}$ generates a $C_{0}$-semi-group of contractions $T(t)=e^{t\mathcal{A}}$ on $\mathcal{E}$.
\end{theorem}

\begin{proof}
Since $\mathcal{D}(\mathcal{A})$ is dense in $\mathcal{E}$, from the
propositions \ref{p7} and \ref{p8} $\mathcal{A}$ is dissipative, maximal,
and by the Lumer-Phillips theorem $\mathcal{A}$ generates a $C_{0}$ semi-group of contractions in $\mathcal{E}$.
\end{proof}
  \begin{lemma}\label{L2} 
  The operator $\mathcal{F}: \mathcal{X} \longrightarrow \mathcal{E}$ given by
  $$\mathcal{F}(U)=\left( 
\begin{array}{c}
wv+wu  \\ 
wv+u \\ 
\end{array}
\right)$$
  is locally Lipschitz.
\end{lemma}
\begin{proof}
    Let $C>0$ and $U_1,U_2\in \mathcal{X}$ such that $\|U_1\|_{\mathcal{X}}\leq C $ and $\|U_2\|_{\mathcal{X}}\leq C$, then 
\begin{equation*}
\begin{array}{lcl}
            \|\mathcal{F}(U_1)-\mathcal{F}(U_2)\|^2_{\mathcal{E}}&=& \int_{\Omega}(w_1v_1-w_2v_2+u_1w_1-u_2w_2)^2+\int_{\Omega}(w_1u_1-w_2u_2+u_1-u_2)^2.
             
\end{array}
\end{equation*} 

Firstly, we analyze the term
\begin{equation*}
    \begin{aligned}
      \int_{\Omega}(w_1v_1-w_2v_2+u_1w_1-u_2w_2)^2 &\leq  2\left(\int_{\Omega}(w_1v_1-w_2v_2)^2+\int_{\Omega}(u_1w_1-u_2w_2)^2\right),
     \end{aligned}
\end{equation*}

Next, we estimate 
\begin{equation}
\begin{aligned}
     \int_{\Omega}(u_1w_1-u_2w_2)^2&=\int_{\Omega}((u_1-u_2)w_1+u_2(w_1-w_2))^2, \\
     &\leq 2\left(\int_{\Omega}((u_1-u_2)w_1)^2+\int_{\Omega}u_2(w_1-w_2)^2\right),\\
     &\leq 2\left(\int_{\Omega}((u_1-u_2)w_1)^2+\int_{\Omega}u_2(w_1-w_2)^2\right),\\
    & \leq 2\left(\|w_{1}\|^2_{\infty}\int_{\Omega}(u_1-u_2)^2+\|w_1-w_2\|^2_{\infty}\int_{\Omega}u_2^2\right),\\
\end{aligned}
    \end{equation}
  
By Poincar\'{e}'s inequality and the continuity of $ H^2(\Omega) \hookrightarrow L^\infty(\Omega) $, there exists a positive constant $ C(\Omega) $ such that
\begin{equation}\label{l1}
    \begin{aligned}
      \int_{\Omega}(u_1w_1-u_2w_2)^2 &\leq  2C(\Omega)\left( \|w_1\|^2_{H^2_0(\Omega)} \|u_2 - u_1\|^2_{L^2(\Omega)}+\|w_1-w_2\|^2_{H^2_0(\Omega)} \|u_2\|^2_{L^2(\Omega)}\right),\\
      &\leq 2C^2C(\Omega)\left( \|u_2 - u_1\|^2_{L^2(\Omega)}+\|w_1-w_2\|^2_{H^2_0(\Omega)}\right).
     \end{aligned}
\end{equation}  

Similarly, we have
\begin{equation}\label{l2}
    \int_{\Omega}(w_1v_1-w_2v_2)^2 \leq 2C^2C(\Omega)\left( \|v_2 - v_1\|^2_{L^2(\Omega)}+\|w_1-w_2\|^2_{H^2_0(\Omega)}\right).
\end{equation}

Combining \eqref{l1} and \eqref{l2}, we obtain

\begin{equation}\label{ist3}
\begin{aligned}
      \int_{\Omega}(w_1v_1-w_2v_2+u_1w_1-u_2w_2)^2 &\leq  8C^2C(\Omega)\left(\|u_2 - u_1\|^2_{L^2(\Omega)}+ \|v_2 - v_1\|^2_{L^2(\Omega)}+\|w_1-w_2\|^2_{H^2_0(\Omega)}\right),\\
      &\leq 8C^2C(\Omega)\|U_1-U_2\|_{\mathcal{X}}^2.
\end{aligned}
    \end{equation}
  
Secondly, we estimate
\begin{equation}\label{ist4}
    \begin{aligned}
      \int_{\Omega}(w_1v_1-w_2v_2+u_1-u_2)^2&\leq  2\left(\int_{\Omega}(w_1v_1-w_2v_2)^2+\int_{\Omega}(u_1-u_2)^2\right),\\
      &\leq 2\left(2C^2C(\Omega)\left( \|v_2 - v_1\|^2_{L^2(\Omega)}+\|w_1-w_2\|^2_{H^2_0(\Omega)}\right)+\|u_2 - u_1\|^2_{L^2(\Omega)}\right),\\
      &\leq 2M\|U_1-U_2\|_{\mathcal{X}}^2,
     \end{aligned}
\end{equation}

where $M=max(1,2C^2C(\Omega)).$

Combining the results \eqref{ist3} and \eqref{ist4}, we find
\begin{equation*}
\begin{array}{lcl}
            \|\mathcal{F}(U_1)-\mathcal{F}(U_2)\|^2_{\mathcal{E}}&\leq& L^2(C)\|U_1-U_2\|_{\mathcal{X}}^2,
             \end{array}
\end{equation*} 
where $L^2(C)=max(2M,8C^2 C(\Omega)).$
Therefor $$\|\mathcal{F}(U_1)-\mathcal{F}(U_2)\|_{\mathcal{E}}\leq L(C)\|U_1-U_2\|_{\mathcal{X}}.$$
This proves that $\mathcal{F}$ is locally Lipschitz continuous.
\end{proof}
\begin{theorem}
    If $ V_0 \in \mathcal{E} $, there is a $t_{max} \leq +\infty$ such that the initial value problem \eqref{ex2} has a unique solution $U\in \mathcal{C} ([0, t_{ max} [, \mathcal{X})$. 
\end{theorem}
\begin{proof}
By Propositions \ref{p5}, \ref{p6}, Lemma \ref{L2} and Theorem \ref{t9}, we have the following: the operator $\mathcal{L}$ is continuous, linear, and invertible, $\mathcal{G}$ is a continuous linear operator, $\mathcal{F}$ is locally Lipschitz, and $\mathcal{A}$ is the infinitesimal generator of a $C_0$-semi-group of contractions $T(t)$. Then, by Theorem \ref{t5}, there exists $t_{max} \leq +\infty$ such that the initial value problem \eqref{eq6} has a unique solution $U\in \mathcal{C} ([0, t_{ max} [, \mathcal{X})$, and this solution satisfies the problem \eqref{ex2}. Therefore, $U$ is also a solution to \eqref{ex2}.
\end{proof}

\end{document}